\documentclass[12pt]{article}

\usepackage{amssymb}
\usepackage{amsthm}
\usepackage{amsmath}

\newtheorem{proposition}{Proposition}[section]

\newtheorem{remark}{Remark}[section]

\DeclareSymbolFont{AMSb}{U}{msb}{M}{n}
\DeclareMathSymbol{\bE}{\mathbin}{AMSb}{"45}
\DeclareMathSymbol{\bF}{\mathbin}{AMSb}{"46}
\DeclareMathSymbol{\bG}{\mathbin}{AMSb}{"47}
\DeclareMathSymbol{\bH}{\mathbin}{AMSb}{"48}
\DeclareMathSymbol{\N}{\mathbin}{AMSb}{"4E}
\DeclareMathSymbol{\PR}{\mathbin}{AMSb}{"50}
\DeclareMathSymbol{\R}{\mathbin}{AMSb}{"52}
\DeclareMathSymbol{\bX}{\mathbin}{AMSb}{"58}
\DeclareMathSymbol{\bY}{\mathbin}{AMSb}{"59}
\DeclareMathSymbol{\bZ}{\mathbin}{AMSb}{"5A}

\newcommand{\half}{{\frac{1}{2}}}

\newcommand{\ca}{{\cal A}}
\newcommand{\cd}{{\cal D}}
\newcommand{\ccf}{{\cal F}}

\newcommand{\cp}{{\cal P}}

\newcommand{\cx}{{\cal X}}
\newcommand{\cy}{{\cal Y}}

\newcommand{\cdal}{{\cd_{\alpha}}}
\newcommand{\cdmal}{{\cd_{-\alpha}}}
\newcommand{\cdmo}{{\cd_{-1}}}
\newcommand{\cdpo}{{\cd_1}}

\newcommand{\phialp}{\phi_\alpha}
\newcommand{\Phialp}{\Phi_\alpha}
\newcommand{\phimalp}{\phi_{-\alpha}}
\newcommand{\Phimalp}{\Phi_{-\alpha}}
\newcommand{\phie}{\phi_e}
\newcommand{\Phie}{\Phi_e}
\newcommand{\phim}{\phi_m}
\newcommand{\Phim}{\Phi_m}
\newcommand{\phimo}{\phi_{-1}}
\newcommand{\Phimo}{\Phi_{-1}}
\newcommand{\phipo}{\phi_1}
\newcommand{\Phipo}{\Phi_1}
\newcommand{\phibet}{\phi_\beta}

\newcommand{\Util}{{\tilde{U}}}
\newcommand{\Vtil}{{\tilde{V}}}

\newcommand{\Pihat}{{\hat{\Pi}}}

\newcommand{\bfe}{{\bf e}}
\newcommand{\bfi}{{\bf i}}
\newcommand{\bfl}{{\bf l}}
\newcommand{\bfm}{{\bf m}}
\newcommand{\bfn}{{\bf n}}
\newcommand{\bfu}{{\bf u}}

\newcommand{\bfP}{{\bf P}}

\newcommand{\bfU}{{\bf U}}
\newcommand{\bfv}{{\bf v}}
\newcommand{\bfw}{{\bf w}}

\newcommand{\bfV}{{\bf V}}
\newcommand{\bfW}{{\bf W}}
\newcommand{\bfy}{{\bf y}}

\newcommand{\abar}{{\bar{a}}}
\newcommand{\fbar}{{\bar{f}}}
\newcommand{\Hbar}{{\bar{H}}}
\newcommand{\bFbar}{{\bar{\bF}}}
\newcommand{\bGbar}{{\bar{\bG}}}
\newcommand{\Mbar}{{\bar{M}}}
\newcommand{\Nbar}{{\bar{N}}}

\newcommand{\Pbar}{{\bar{P}}}
\newcommand{\Ubar}{{\bar{U}}}
\newcommand{\phibar}{{\bar{\phi}}}
\newcommand{\Phibar}{{\bar{\Phi}}}

\newcommand{\bfubar}{{\bar{\bfu}}}
\newcommand{\bfUbar}{{\bar{\bfU}}}
\newcommand{\bfvbar}{{\bar{\bfv}}}
\newcommand{\bfVbar}{{\bar{\bfV}}}
\newcommand{\bfwbar}{{\bar{\bfw}}}

\newcommand{\E}{{\hbox{\bf E}}}
\DeclareMathOperator{\id}{{id}}
\newcommand{\Emu}{{\hbox{\bf E}_\mu}}
\newcommand{\LL}{{L^2(\mu)}}
\newcommand{\EP}{{\hbox{\bf E}_P}}

\newcommand{\fndot}{{\,\cdot\,}}

\title{Manifolds of Differentiable Densities
  \thanks{To appear in ESAIM: Probability and Statistics. The original publication is
  available at www.esaim-ps.org \copyright\ EDP Sciences, SMAI}}

\author{Nigel J.~Newton
  \thanks{School of Computer Science and Electronic Engineering, University of Essex,
  Wivenhoe Park, Colchester, CO4 3SQ, United Kingdom. njn@essex.ac.uk}}

\begin{document}

\maketitle

\begin{abstract}
We develop a family of infinite-dimensional (non-parametric) manifolds of probability
measures.  The latter are defined on underlying Banach spaces, and have densities of
class $C_b^k$ with respect to appropriate reference measures.  The case $k=\infty$, in
which the manifolds are modelled on Fr\'{e}chet spaces, is included.  The manifolds
admit the Fisher-Rao metric and, unusually for the non-parametric setting, Amari's
$\alpha$-covariant derivatives for all $\alpha\in\R$.  By construction, they are
$C^\infty$-embedded  submanifolds of particular manifolds of finite measures.  The
statistical manifolds are dually ($\alpha=\pm 1$) flat, and admit mixture and exponential
representations as charts.  Their curvatures with respect to the $\alpha$-covariant
derivatives are derived.  The likelihood function associated with a finite sample is a
continuous function on each of the manifolds, and the $\alpha$-divergences are of
class $C^\infty$.


\end{abstract}

\section{Introduction} \label{se:intro}

{\em Information Geometry} is the study of differential-geometric structures arising in
the theory of statistical estimation, and has a history going back (at least) to the work
of C.R.~Rao \cite{rao1}.  It is finding increasing application in many fields including
{\em asymptotic statistics}, {\em machine learning}, {\em signal processing} and
{\em statistical mechanics}. (See, for example, \cite{niba1,niba2} for some recent
developments.) The theory in finite dimensions (the parametric case) is well developed,
and treated pedagogically in a number of texts \cite{amar1,barn1,chen1,laur1,muri1}.  A
classical example is the finite-dimensional {\em exponential model}, in which linear
combinations of a finite number of real-valued random variables (defined on an underlying
probability space $(\bX,\cx,\mu)$) are exponentiated and normalised to generate
probability density functions with respect to the reference measure $\mu$.  The topology
induced on the set of probability measures, thus defined, is consistent with the
important statistical divergences of estimation theory, and derivatives of the latter can
be used to define geometric objects such as a Riemannian metric (the Fisher-Rao metric)
and a family of covariant derivatives.

Central to any extension of these ideas to infinite dimensions, is the use of charts with
respect to which statistical divergences are sufficiently smooth.  The
{\em Kullback-Leibler divergence} between two probability measures $P\ll Q$ is defined
as follows:
\begin{equation}
\cd_{\rm KL}(P\,|\,Q) := \E_Q(dP/dQ)\log(dP/dQ), \label{eq:kldiv}
\end{equation}
where $\E_Q$ represents expectation (integration) with respect to $Q$. As is clear from
(\ref{eq:kldiv}), the regularity of $\cd_{\rm KL}$ is closely connected with that of
the density, $dP/dQ$, and its log (considered as elements of dual spaces of
real-valued functions on $\bX$).  In fact, much of information geometry concerns the
interplay between these two representations of $P$, and the exponential map that connects
them.  The two associated affine structures form the basis of a Fenchel-Legendre transform
underpinning the subject, and so manifolds that fully accommodate these structures are
particularly useful.

In the series of papers \cite{cepi1,gipi1,piro1,pise1}, G.~Pistone and his co-workers
developed an infinite-dimensional variant of the exponential model outlined above.
Probability measures in the manifold are mutually absolutely continuous with respect to
the reference measure $\mu$, and the manifold is covered by the charts
$s_Q(P)=\log dP/dQ-\E_Q\log dP/dQ$ for different ``patch-centric'' probability measures
$Q$.  These readily give $\log dP/dQ$ the desired regularity, but require exponential
Orlicz model spaces in order to do the same for $dP/dQ$.  The exponential Orlicz
manifold, $M_O$, has a strong topology, under which $\cd_{\rm KL}$ is of class $C^\infty$.
In \cite{piro1}, the authors define ``mean parameters'' on $M_O$.  Like $s_Q$, these
are defined locally at each $Q\in M_O$, by $\eta_Q(P)=dP/dQ-1$.  $\eta_Q$ maps into the
``pre-dual'' of the exponential Orlicz space.  However, despite being injective, $\eta_Q$
is not homeomorphic and so cannot be used as a chart.  By contrast, the manifolds
developed in \cite{newt4,newt7} use the ``balanced'' global chart
$\phi(P)=dP/d\mu-1+\log dP/d\mu-\Emu\log dP/d\mu$, thereby enabling the use of model
spaces with weaker topologies.  (In order for $\cd_{\rm KL}$ to be of class $C^k$, it
suffices to use the Lebesgue model space $L^p(\mu)$ with $p=k+1$.)  The balanced $L^p$
manifold, $M_B$, admits {\em mixture} and {\em exponential} representations
$m,e:M_B\rightarrow L^p(\mu)$, defined by $m(P)=dP/d\mu-1$ and
$e(P)=\log dP/d\mu-\Emu \log dP/d\mu$.  Like $\eta_Q$ on $M_O$, these are injective but
not homeomorphic, and so cannot be used as charts.  The Hilbert case, in which $p=2$, is
developed in detail in \cite{newt4}.

The exponential Orlicz and balanced $L^p$ manifolds (for $p\ge 2$) support the
infinite-dimensional variant of the Fisher-Rao metric, and (for $p\ge 3$) the
infinite-dimensional variant of the {\em Amari-Chentsov tensor}.  The latter can be used
to define {$\alpha$-derivatives} on particular {\em statistical bundles} \cite{gipi1}.
However, with the exception of the case $\alpha=1$ on the exponential Orlicz manifold,
these bundles differ from the tangent bundle, and so the $\alpha$-derivatives do not
constitute covariant derivatives in the usual sense.  This fact is intimately
connected with the non-homeomorphic nature of $\eta_Q$ on $M_O$, and $m$ and $e$ on $M_B$.

In \cite{ajls1}, the authors define a very general notion of {\em statistical model}.
This is a manifold equipped with a metric and symmetric 3-tensor, together with an
embedding into a space of finite measures, such that these become the Fisher-Rao metric
and Amari-Chentsov tensor.  They extend a result of Chentsov (on the uniqueness of
these tensors as invariants under sufficient statistics) to this much wider class of
statistical models.  The exponential Orlicz and balanced $L^p$ manifolds (for $p\ge 3$)
all fit within this framework.

The topologies of these manifolds (like those of all manifolds of ``pure'' information
geometry) have no direct connection with any topology that the underlying space
$(\bX,\cx,\mu)$ may possess.  They concern statistical inference in its barest form --
statistical divergences measure dependency between random variables without recourse
to structures in their range spaces any richer than a $\sigma$-algebra of events.
Nevertheless, metrics, topologies and linear structures on $\bX$ play important roles
in many applications.  In {\em maximum likelihood estimation}, for example, it is
desirable for the likelihood function associated with a finite sample to be continuous.
It is, therefore, of interest to develop statistical manifolds that embrace both
topologies.  This is a central aim here; we incorporate the topology of $\bX$ by using
appropriate model space norms. A different approach is pursued in \cite{fuku1}. The
exponential manifolds developed there admit, by construction, continuous evaluation maps
(such as the likelihood function) since they are based on reproducing kernel Hilbert
space methods.  However, like the exponential Orlicz and balanced $L^p$ manifolds, they
do not fully accommodate the affine structure associated with the density.

In developing this work, the author was motivated by problems in Bayesian estimation, in
which posterior distributions must be computed from priors and partial observations.
Suppose, for example, that $X:\Omega\rightarrow\bX$ and $Y:\Omega\rightarrow\bY$ are
random variables defined on a common probability space $(\Omega,\ccf,\PR)$, and taking
values in metric spaces $\bX$ and $\bY$, respectively.  Let $\cx$ be the
$\sigma$-algebra of Borel subsets of $\bX$, and let $\cp(\cx)$ be the set of probability
measures on $\cx$.  Under mild conditions we can construct a {\em regular conditional
probability distribution} for $X$ given $Y$, $\Pi:\bY\rightarrow\cp(\cx)$.  (See, for
example, \cite{lish1}.)  This has the key properties that, for each $B\in\cx$,
$\Pi(\cdot)(B):\bY\rightarrow [0,1]$ is measurable and $P(X\in B|Y)=\Pi(Y)(B)$.  In
many instances of this problem, the image $\Pi(\bY)$ is contained in an
infinite-dimensional statistical manifold, and particular statistical divergences
defined on the manifold can be interpreted as ``multi-objective'' measures of error in
approximations to $\Pi$.  For example, Pearson's $\chi^2$-divergence can be interpreted
as a multi-objective measure of the normalised mean-squared errors in estimates of
real-valued variates, $f(X)$ \cite{newt7}.
\begin{eqnarray}
\cd_{\chi^2}(P\,|\,Q)
  & := & \half\|dP/dQ-1\|_{L_0^2(Q)}^2
         = \half\sup_{f\in F}(\E_Q (dP/dQ-1)f)^2 \nonumber \\
         [-1.5ex] \label{eq:chisqd} \\ [-1.5ex]
  &  = & \half\sup_{f\in L^2(Q)}\frac{(\E_P f-\E_Q f)^2}{\E_Q(f-\E_Q f)^2}, \nonumber
\end{eqnarray}
where $L_0^2(Q)$ is the set of $Q$-square integrable, real-valued functions having
mean zero, and $F$ is the subset of those functions having unit variance. 
If $\Pihat:\bY\rightarrow\cp(\cx)$ is used as an approximation to $\Pi$, and
$\E_{\Pihat(Y)}f$ as an approximation of the variate $f(X)$, then the mean-squared error
admits the orthogonal decomposition:
\begin{equation}
\bE\left(f(X)-\E_{\Pihat(Y)}f\right)^2
  = \bE \E_{\Pi(Y)}\left(f-\E_{\Pi(Y)}f\right)^2
    + \bE\left(\E_{\Pihat(Y)}f-\E_{\Pi(Y)}f\right)^2. \label{eq:l2decom}
\end{equation}
The first term on the right-hand side here is the {\em statistical error} arising from
the limitations of the observation $Y$, whereas the second term is the approximation
error arising from the use of $\Pihat$ instead of $\Pi$. Since there is no point in
approximating $\E_{\Pi(Y)}f$ with great accuracy if it is itself a poor estimate of $f(X)$, it is appropriate to consider the magnitude of the second term {\em relative}
to that of the first.   As the final term in (\ref{eq:chisqd}) shows, the
$\chi^2$-divergence, $\cd_{\chi^2}(\Pihat(Y)\,|\,\Pi(Y))$, selects the worst of these
relative errors among the square-integrable variates $f(X)$.  The divergences
$\cd_{KL}$ and $\cd_{\chi^2}$ are both members of the one-parameter family of
{\em $\alpha$-divergences}, $(\cd^\alpha,\alpha\in\R)$ \cite{amar1}.  (In fact
$\cd_{KL}=\cd^{-1}$ and $\cd_{\chi^2}=\cd^{-3}$.)

Bayesian problems, in which a Markov process $(X_t,0\le t<\infty)$ has to be estimated
at each time $t$, on the basis of the history of an observation process
$(Y_s,0\le s\le t)$, are known as problems in {\em nonlinear filtering}. Regular
conditional distributions for $X$ are then time dependent, and can often be represented
as solutions of stochastic partial differential equations \cite{lish1}.  Suppose, for
example, that $X$ is an $\R^d$-valued diffusion process with drift coefficient
$b:\R^d\rightarrow\R^d$ and positive semi-definite diffusion coefficient
$a:\R^d\rightarrow\R^{d\times d}$.  Suppose, further, that $Y$ is a real-valued
partial observation process of the type
\begin{equation}
Y_t = \int_0^t h(X_s)\,ds + W_t, \quad 0\le t<\infty ,\label{eq:obsproc}
\end{equation}
where $h:\R^d\rightarrow\R$ is a measurable function and $W$ is a Brownian motion,
independent of $X$.  Under suitable regularity constraints on $a$, $b$ and $h$, the
$(Y_s,0\le s\le t)$-conditional distribution of $X_t$ has a density $\pi_t$ satisfying
the {\em Kushner-Stratonovich equation} \cite{lish1}:
\begin{equation}
\pi_t = \pi_0 + \int_0^t \ca\pi_s\,ds + \int_0^t \pi_s(h-\bar{h}_s)(dY_s-\bar{h}_s\,ds),
        \label{eq:kusstra}
\end{equation}
where $\ca$ is the Kolmogorov forward (Fokker-Planck) operator for $X$, and $\bar{h}_s$
is the conditional mean of $h$:
\begin{equation}
\ca f = \half\sum_{i,j=1}^d\frac{\partial^2(a_{ij}f)}{\partial x_i\partial x_j}
        - \sum_{i=1}^d\frac{\partial(b_i f)}{\partial x_i}\quad{\rm and}\quad
\bar{h}_s = \int h(x)\pi_s(x)\,dx. \label{eq:forobs}
\end{equation}

If posterior distributions of a nonlinear filter stay on a statistical manifold, then it
is possible to use the methods of information geometry to study its information-theoretic
properties, and to develop approximations based on finite-dimensional submanifolds.
These ideas are developed in \cite{brpi1} in the context of the exponential Orlicz
manifold, and in \cite{newt5} in the context of the balanced Hilbert manifold.  However
these manifolds are not suited to the quest for infinite dimensional evolution equations
since they are constructed without reference to the topology of $\R^d$, which is clearly
needed in any vector field representation of the first integral in (\ref{eq:kusstra}).
To overcome this problem, we need a statistical manifold with a model space whose members
satisfy suitable differentiability constraints.  In \cite{lopi1}, the authors define an
infinite-dimensional statistical manifold modelled on a weighted Orlicz-Sobolev space,
and use it to study the Boltzmann equation.  This manifold can also be used to study weak
solutions of (\ref{eq:kusstra}).

The two integral terms on the right-hand side of (\ref{eq:kusstra}) are {\em mixture
affine} (respectively {\em exponential affine}), in the sense that the integrands are
affine maps in the $\eta_Q$/$m$ (respectively $s_Q$/$e$) representations.  In the quest
for evolution equations, it is therefore advantageous to use a manifold that admits both
these representations as charts.  Such a manifold is constructed here;  it comprises
probability measures whose log-densities with respect to a reference measure are of
class $C_b^k$.  In particular, their densities have strictly positive infima, which is
a significant restriction if, for example, $\bX=\R^d$.  This constraint can be
thought of as an infinite-dimensional equivalent of the positivity constraint placed on
the probabilities of all atoms in the finite sample space setting.  (See section 2.5
in \cite{amar1}.) As in that setting, the removal of this constraint would add a boundary
to the manifold, on which certain divergences are singular.  The manifolds constructed
are suited to nonlinear filtering problems in which the process $X$ is constrained to lie
in a bounded subset of $\R^d$ by certain types of boundary condition, such as reflective
(Neumann) or periodic conditions.  (See Remark \ref{re:variants}(iv), below.)  If the
boundaries are sufficiently far from the origin, then problems with boundaries may be
just as good models for physical systems as those without.  If, for example, the drift
and diffusion coefficients $b$ and $a$ are bounded, and $a$ is strictly positive definite,
then posterior densities are known to have Gaussian tails.  Such a density may be no more
accurate a representation of reality 10 standard deviations from its mean than a
density coming from a model that incorporates a reflective boundary at this point. 

Aside from applications to nonlinear filtering, it is of fundamental interest to develop
non-parametric statistical manifolds that admit the full geometry of Amari's
$\alpha$-covariant derivatives---something that is not achieved in the manifolds
described above.

The paper is structured as follows.  Sections \ref{se:modspa} and \ref{se:manfin}
construct $M$, a smooth manifold of {\em finite} measures on an open subset of a
{\em Banach space} $\bX$, whose densities with respect to a reference measure are of
class $C_b^k$.  $M$ is covered by each chart in a one-parameter family
$(\phialp, \alpha\in\R)$.  The charts $\phialp$ and $\phimalp$ map to open subsets of
the affine spaces of a Fenchel-Legendre transform involving the $\alpha$-divergences
$\cdal$ and $\cdmal$.  As such, they define a metric and dual notions of parallel
transport on the tangent bundle for each $\alpha\in\R$.  Since all these charts are
global, $M$ is flat with respect to the associated covariant derivatives,
$(\nabla^\alpha,\alpha\in\R)$.  Section \ref{se:manprob} considers the subset of
{\em probability} measures, $N$.  This is a $C^\infty$-embedded submanifold of $M$,
from which it inherits its important properties.  In particular, the projection of the
metric and covariant derivatives of $M$ onto $N$ yields the Fisher-Rao metric, and the
$\alpha$-covariant derivatives on $N$.  In contrast with the manifolds of
\cite{cepi1,gipi1,lopi1,newt4,newt7,piro1,pise1}, the latter are all defined on the tangent
bundle of $N$. Of course, this extra regularity is gained at the cost of inclusiveness.
$N$ is (dually) flat in the $\alpha = \pm 1$-covariant derivatives, and admits affine
mixture and exponential charts.  Finally, section \ref{se:smooth} uses the method
of projective limits to extend these results to manifolds of {\em smooth} densities.
The idea of embedding non-parametric statistical manifolds in manifolds of finite
measures is not new. (See, for example, \cite{ajls1} and \cite{newt7}.)  However, the
fact that the ambient manifold here admits a multiplication operator ((\ref{eq:multop})
below) allows the global $\alpha$-geometry of the statistical manifold to be obtained
in an extrinsic manner.

In recent related work \cite{babm1,brmi1}, the authors construct a manifold of smooth
densities on an underlying finite-dimensional manifold by considering such densities to
be smooth sections of the associated {\em volume bundle}. (This is a vector bundle of
dimension 1 that endows the underlying manifold with an intrinsic notion of volume.)
They consider a property of invariance of Riemannian metrics under the diffeomorphism
group of the underlying manifold, and construct the class of all metrics with this
property.  When restricted to the submanifold of {\em probability measures}, these
all coincide (modulo scaling) with the Fisher-Rao metric.  In \cite{brmi1}, they
develop the Levi-Civita covariant derivative and carry out a number of extensions and
completions of the manifold in order to study its global geometry. 

\section{The exponential map} \label{se:modspa}

Let $B$ be an open subset of a Banach space $\bX$, and let
$\cx_B:=\{A\subset B: A\in\cx\}$, where $\cx$ is the Borel $\sigma$-algebra on $\bX$.
Let $\mu$ be a probability measure on $(B,\cx_B)$ with the following property: for any
non-empty open $A\in\cx_B$, $\mu(A)>0$.  (For example, $\bX=\R^d$, $B$ is a bounded open
rectangle, and $\mu$ is normalised Lebesgue measure.)  Let $\bG:=C^k_b(B;\R)$ be the space
of continuous and bounded functions $a:B\rightarrow \R$, that have continuous and bounded
(Fr\'{e}chet) derivatives of all orders up to some $k\in\N_0$. $\bG$ is a Banach space
over $\R$ when endowed with the norm:
\begin{equation}
\|a\|_\bG = \sup_{x\in B}|a(x)| + \sum_{i=1}^k \sup_{x\in B}\|a^{(i)}_x\|_{L(\bX^i;\R)},
            \label{eq:gnorm}
\end{equation}
where $a^{(i)}:B\rightarrow L(\bX^i;\R)$ is the $i$'th derivative of $a$, and
$L(\bX^i;\R)$ is the Banach space of continuous multilinear functions from $\bX^i$ to
$\R$, endowed with the operator norm.  The (continuous bilinear) {\em multiplication
operator} $\pi:\bG\times\bG\rightarrow\bG$, and the (continuous linear) {\em expectation
functional} $\Emu:\bG\rightarrow\R$, are as follows
\begin{equation}
(a\cdot b)(x)
  = \pi(a,b)(x) = a(x)b(x)\quad{\rm and}\quad \Emu a=\int_Ba(x)\mu(dx). \label{eq:multop}
\end{equation}
Equipped with $\pi$, $\bG$ becomes a commutative Banach algebra with identity
$\bfe\equiv 1$.  In the special case that $k=0$, it is a commutative $C^*$-algebra with
involution the identity map.

\begin{proposition} \label{pr:nemexp}
The Nemytskii (superposition) operator, $\exp_\bG:\bG\rightarrow \bG^+$, defined by
$\exp_\bG(a)(x)=\exp_\R(a(x))$, is diffeomorphic onto its image
$\bG^+:=\{a\in\bG: \inf_{x\in B}a(x)>0\}$, and has first derivative
\begin{equation}
\exp_{\bG,a}^{(1)}u = \exp_\bG(a)\cdot u.  \label{eq:expder}
\end{equation}
\end{proposition}

\begin{proof}
Let $F:\bG\times\bG\rightarrow\bG$ be defined by
$F(a,b)=\exp_\bG(b)-\exp_\bG(a)-\exp_\bG(a)\cdot(b-a)$.  In order to prove
(\ref{eq:expder}) it suffices to show that, for any $a\in\bG$, there exists a
$K_a<\infty$ such that
\begin{equation}
\|F(a,b)\|_\bG \le K_a\|b-a\|_\bG^2 \quad{\rm for\ all\ }b\in B(a,1), \label{eq:Fbnd}
\end{equation}
where $B(a,1)$ is the open unit ball centred at $a$.  That this is so when $k=0$
follows from an application of Taylor's theorem to $\exp_\R$.  Suppose that $k\ge 1$.
Fixing $a\neq b\in\bG$, and differentiating $F(a,b)$ with respect to $x$, we obtain
\begin{equation}
F(a,b)_x^{(1)}y = F(a,b)(x)b_x^{(1)}y + H(a,b,x)y, \label{eq:Fderiv}
\end{equation}
where $H:\bG\times\bG\times B\rightarrow L(\bX;\R)$ is defined by
\[
H(a,b,x)y = (\exp_\bG(a)\cdot(b-a))(x)(b_x^{(1)}-a_x^{(1)})y.
\]
An induction argument, starting from (\ref{eq:Fderiv}), shows that, for any $1\le i\le k$,
\begin{equation}
F(a,b)_x^{(i)}\bfy_1^i
  = \sum_{\rho\in S_i}\sum_{j=1}^i\gamma_{i,\rho,j}
      F(a,b)_x^{(i-j)}\bfy_{\rho_{j+1}}^{\rho_i} b_x^{(j)}\bfy_{\rho_1}^{\rho_j}
      + H(a,b,\fndot)_x^{(i-1)}\bfy_1^i, \label{eq:Fexp}
\end{equation}
where $\bfy_m^n:=(y_m,\ldots, y_n)$, $S_i$ is the set of all permutations of the integers
1 to $i$, and  the real constants $\gamma_{i,\rho,j}$ are defined by the following
recursion: $\gamma_{i,\rho,0}=\gamma_{i,\rho,i+1}=0$, $\gamma_{1,\bfe,1}=1$, and for any
$\sigma\in S_{i+1}$ and any $1\le j\le i+1$,
\[
\gamma_{i+1,\sigma,j}
  = \left\{\begin{array}{ll}
      \gamma_{i,\rho,j} \quad & {\rm if}\; \sigma=(\rho,i+1){\rm\; for\ some\ }\rho\in S_i \\
      \gamma_{i,\rho,j-1} & {\rm if}\; \sigma=
         (\rho_1,\cdots,\rho_{j-1},i+1,\rho_j,\cdots, \rho_i)
         {\rm\; for\ some\ }\rho\in S_i \\
      0 & {\rm otherwise}.
      \end{array}\right.
\] 
For any $a\in\bG$, there exists a $K_a<\infty$ such that
\begin{equation}
\sup_{x\in B}\|H(a,b,\fndot)_x^{(i-1)}\|_{L(\bX^i;\R)} \le K_a\|b-a\|_\bG^2
  \quad{\rm for\ all\ }b\in\bG. \label{eq:Hbnd}
\end{equation}
An induction argument on $i$ thus establishes (\ref{eq:Fbnd}), and hence (\ref{eq:expder}).
A further induction argument readily shows that $\exp_\bG\in C^\infty(\bG;\bG^+)$.

For any $a\in\bG$, the linear map $\exp_{\bG,a}^{(1)}:\bG\rightarrow\bG$ of
(\ref{eq:expder}) is clearly a toplinear isomorphism, and so the statement of the
proposition follows from the inverse mapping theorem.
\end{proof}

\begin{remark} \label{re:variants}
\begin{enumerate}
\item[(i)] The crucial feature of this setup is that $\exp_\bG(\bG)$ is an open subset
of $\bG$, which is essential if both mixture and exponential representations are to be
charts.  This property, which is connected with the existence of the multiplication
operator of (\ref{eq:multop}), does not hold if $\bG$ is replaced by the exponential Orlicz
spaces of \cite{pise1,lopi1}, or the Lebesgue $L^p(\mu)$ spaces of \cite{newt4,newt7}.
(See examples 2.1 and 2.2 in \cite{newt4}.)
\item[(ii)] The model space $\bG$ is based on an {\em open} subset of $\bX$ for reasons
of inclusiveness.  The closure of $B$, $\bar{B}$, does not need to be compact. If,
however, $\bar{B}$ {\em is} compact, then the space
$\{a\in\bG: a=b|_B {\rm\ for\ some\ }b\in C^k(\bar{B};R)\}$ is a proper, closed subspace
of $\bG$, and so defines smoothly embedded submanifolds of those constructed in sections
\ref{se:manfin} and \ref{se:manprob}.
\item[(iii)] The measure $\mu$ on $(B,\cx_B)$ must be {\em finite} in order for the
integral functional of (\ref{eq:multop}) to be well defined.  Since the total mass of
$\mu$ does not affect the results that follow, it is natural to assume that it is $1$.
$\mu$ is then, itself, a member of the statistical manifold of section \ref{se:manprob}.
\item[(iv)] The results that follow hold true in other scenarios.  For example, that in
which $\bX=\R^d$, $B=(-\pi,\pi)^d$, $\mu=(2\pi)^{-d}\hbox{\rm Leb}$ and $\bG$ is the
subspace of $C_b^k(B;\R)$ whose members satisfy a suitable periodic boundary condition.
The manifolds constructed then comprise measures defined on the $d$-dimensional torus.
\item[(v)] $\bG$ can also be replaced by $L^\infty(B;\R)$, but no account is then taken
of the topology of $\bX$.  Cf.~\cite{loqu1}, in which the authors follow the approach of
\cite{pise1} to construct a {\em Tsallis $q$-exponential} statistical manifold modelled
on a Banach space of essentially bounded functions.
\end{enumerate}
\end{remark}

\section{The manifold of finite measures} \label{se:manfin}

Let $M$ be the set of finite measures on $(B,\cx_B)$ that are mutually absolutely
continuous with respect to $\mu$, and have densities of the form
\begin{equation}
dP/d\mu = \exp_\bG(a), \quad{\rm for\ some\ }a\in\bG. \label{eq:measure}
\end{equation}
$M$ is covered by the single chart $\phipo:M\rightarrow\bG$, defined by
\[
\phipo(P)=\log_{\bG^+}(dP/d\mu):=\exp_\bG^{-1}(dP/d\mu).
\]
A tangent vector at $P\in M$ is a signed measure on $(B,\cx_B)$ that is absolutely
continuous with respect to $\mu$, and has a density with respect to $P$ of the form
\begin{equation}
dU/dP = u, \quad{\rm for\ some\ }u\in\bG.  \label{eq:tanvec}
\end{equation}
The tangent space at $P$, $T_PM$, is the linear space of all such measures, and the tangent
bundle is the disjoint union $TM:=\bigcup_{P\in M}(P, T_PM)$.  By construction, $TM$ is
globally trivialised by the bijection $\Phipo:TM\rightarrow\bG\times\bG$ defined by
\begin{equation}
\Phipo(P,U) = \left(\log_{\bG^+} dP/d\mu ,\, dU/dP\right). \label{eq:Phidef}
\end{equation}
The derivative of a differentiable, Banach-space-valued map $f:M\rightarrow \bY$
(at $P$ and in the direction $U\in T_PM$) is defined in the obvious way:
\begin{equation}
Uf := (f\circ\phipo^{-1})_a^{(1)}u, \quad{\rm where\ }(a,u)=\Phipo(P,U). \label{eq:tangop}
\end{equation}
For any $\alpha\in\R\setminus\{1\}$, let
$\phialp:M\rightarrow\bG$ be defined as follows:
\begin{equation}
\textstyle \phialp(P)
  = \frac{2}{1-\alpha}\left(\exp_\bG\left(\frac{1-\alpha}{2}\phipo(P)\right)-1\right).
    \label{eq:phialp}
\end{equation}
Proposition \ref{pr:nemexp} shows that the map $\phialp\circ\phipo^{-1}$ is
diffeomorphic onto its image, and so $(\phialp,\; \alpha\in\R)$ is a smooth atlas,
each chart of which covers $M$.  For any $\alpha,\beta\in\R$, the derivative of the
transition map $\phialp\circ\phibet^{-1}$ is,
\begin{equation}
\textstyle(\phialp\circ\phibet^{-1})_a^{(1)}u
  = \exp_\bG\left(\frac{\beta-\alpha}{2}a_1\right)\cdot u, \label{eq:transmap}
\end{equation}
where $a_1=\phipo\circ\phibet^{-1}(a)$.  For each $\alpha\in\R$ the chart
$\Phialp:TM\rightarrow\bG\times\bG$, defined by
\begin{equation}
\Phialp(P,U) = (\phialp(P),U\phialp), \label{eq:Phialpdef}
\end{equation}
induces a distinct global trivialisation of $TM$.  In particular,
$\Phimo(P,U)=(dP/d\mu-1, dU/d\mu)$.

\begin{remark} \label{re:phioff}
\begin{enumerate}
\item [(i)]The maps $\phialp$ are derived from Amari's {\em $\alpha$-embedding maps}.
(See section 2.6 in \cite{amar1}.)  The offset $-1$ is included in (\ref{eq:phialp})
so that $\phialp(\mu)=0$.  This also ensures that
$\phialp\circ\phimo^{-1}\circ(\id_{\bG^+}-1):\bG^+\rightarrow\bG$ is Tsallis'
{\em $q$-deformed logarithm} with $q=(1+\alpha)/2$. (See, for example, chapter 7 of
\cite{naud1}.) It is easily established that, for any fixed $P\in M$, the map
$\R\ni\alpha\mapsto\phialp(P)\in\bG$ is of class $C^\infty$.
\item[(ii)] We introduce multiple charts in order to define different notions of parallel
transport on $TM$.  Maps simlar to $(\phialp, \alpha\in[-1,1])$ are defined and studied
on the exponential Orlicz manifold in \cite{gipi1}, and on the balanced $L^p(\mu)$
manifolds in \cite{newt7}.  However, since they are not diffeomorphic in those contexts,
it is not possible to use them to define parallel transport on the associated tangent
bundles.
\item[(iii)] The charts $\phimo$ and $\phipo$ are particularly important. $\phimo$
reflects the inherent linear structure of a set of measures.  On the other hand, $\phipo$
is surjective, and so trivially introduces a Lie group structure on $M$.  For $P,Q\in M$,
the product $PQ$ and power $P^\lambda$ (for any $\lambda\in\R$) are defined as
follows:
\begin{equation}
\frac{dPQ}{d\mu}
  = \frac{dQP}{d\mu}
  = \left(\frac{dP}{d\mu}\cdot\frac{dQ}{d\mu}\right) \quad{\rm and}\quad
    \frac{dP^\lambda}{d\mu} = \left(\frac{dP}{d\mu}\right)^\lambda, \label{eq:group}
\end{equation}
and the identity is $\mu$. The {\em power of a measure} for $\lambda\in(0,1]$ is
used in \cite{ajls2} to characterise the Fisher-Rao metric and Amari-Chentsov tensor on
parametric statistical manifolds admitting singular measures.  Since the log densities,
here, are bounded, $P^\lambda$ is defined for all real $\lambda$.
\end{enumerate}
\end{remark}

Let $\Gamma TM$ be the space of smooth sections of $TM$ (i.e.~smooth vector fields).
Each chart $\Phialp$ induces a notion of parallel transport on $TM$; tangent vectors in
different fibres of $TM$, $U\in T_PM$ and $\Util\in T_QM$, are $\alpha$-parallel 
transports of each other if $U\phialp=\Util\phialp$.  The associated covariant derivative,
$\nabla^\alpha:\Gamma TM\times\Gamma TM\rightarrow \Gamma TM$, is that for which
$\phialp$ is an affine chart:
\begin{equation}
\nabla_\bfU^\alpha \bfV\phialp = \bfU\bfV\phialp. \label{eq:covder}
\end{equation}
As is the case for all covariant derivatives defined from global affine charts,
$\nabla^\alpha$ is torsion free, and $M$ is $\nabla^\alpha$-flat (or simply
$\alpha$-flat) for all $\alpha\in\R$.  In fact, for any $\bfU,\bfV,\bfW\in\Gamma TM$ and
any $\alpha\in\R$,
\begin{eqnarray}
& & \left(\nabla_\bfU^\alpha \bfV - \nabla_\bfV^\alpha \bfU - [\bfU,\bfV]\right)
    \phialp \nonumber \\
& & \qquad\qquad\qquad = \bfU\bfV\phialp - \bfV\bfU\phialp - [\bfU,\bfV]\phialp = 0,
    \nonumber \\
    [-1.5ex] \label{eq:torcur} \\ [-1.5ex]
& & \left(\nabla_\bfU^\alpha\nabla_\bfV^\alpha\bfW 
      - \nabla_\bfV^\alpha\nabla_\bfU^\alpha\bfW
      - \nabla_{[\bfU,\bfV]}^\alpha\bfW\right)\phialp \nonumber \\
& & \qquad\qquad\qquad  = \bfU\bfV\bfW\phialp - \bfV\bfU\bfW\phialp
      - [\bfU,\bfV]\bfW\phialp = 0, \nonumber
\end{eqnarray}
where $[\cdot,\cdot]$ is the Lie bracket.  $\alpha$-geodesics are curves of $M$ whose
$\phialp$-representations are straight lines in $\bG$.

We define a weak Riemannian metric on $M$ via the inclusion $\bG\subset  L^2(P)$: for
any $U,V\in T_PM$,
\begin{equation}
\langle U, V \rangle_P
  := \langle U\phipo, V\phipo\rangle_{L^2(P)}
   = \langle U\phialp, V\phimalp\rangle_\LL \quad{\rm for\ all\ }\alpha\in\R,
     \label{eq:riemet}
\end{equation}
where we have used (\ref{eq:transmap}) in the second step.  This is positive definite
since $P(A)>0$ for any non-empty open set $A\in\cx_B$.  (It is not a {\em strong}
Riemannian metric since there are members of the cotangent space that do not admit the
representation $\langle U,\cdot\rangle_P:T_PM\rightarrow\R$ for some $U\in T_PM$.)

As is clear from (\ref{eq:riemet}), if $\Util,\Vtil\in T_QM$ are obtained by parallel
transport of $U,V\in T_PM$, one according $\Phialp$ and the other according to $\Phimalp$,
then $\langle\Util,\Vtil\rangle_Q=\langle U,V\rangle_P$.  In this sense $\nabla^\alpha$
and $\nabla^{-\alpha}$ are {\em dual} with respect to the metric.  This can be expressed
in differential form as follows: for any $\bfU,\bfV,\bfW\in\Gamma TM$,
\begin{eqnarray}
\bfU\langle \bfV, \bfW\rangle
  & = & \bfU\langle\bfV\phialp, \bfW\phimalp\rangle_\LL \nonumber \\
  & = & \langle\bfU\bfV\phialp, \bfW\phimalp\rangle_\LL
        + \langle\bfV\phialp, \bfU\bfW\phimalp\rangle_\LL
         \label{eq:covdual} \\
  & = & \langle \nabla_\bfU^\alpha \bfV, \bfW \rangle
        + \langle \bfV, \nabla_\bfU^{-\alpha}\bfW \rangle. \nonumber
\end{eqnarray}
(Cf.~the finite-dimensional case \cite{amar1}.)  Being self-dual and torsion free,
$\nabla^0$ is the Levi-Civita covariant derivative associated with the metric.  The
linear relation between the $\alpha$-covariant derivatives is also retained:
\begin{equation}
\nabla^\alpha
  = \textstyle\frac{1-\alpha}{2}\nabla^{-1} + \frac{1+\alpha}{2}\nabla^1. \label{eq:linrel}
\end{equation}
This follows from (\ref{eq:transmap}), which shows that
\begin{eqnarray*}
\nabla_\bfU^{\pm 1}\bfV\phialp
  & = & (\phialp\circ\phi_{\pm 1}^{-1})_{\phi_{\pm 1}}^{(1)}\nabla_\bfU^{\pm 1}\bfV
        \phi_{\pm 1} \\
  & = & (\phialp\circ\phi_{\pm 1}^{-1})_{\phi_{\pm 1}}^{(1)}
        \bfU\left[(\phi_{\pm 1}\circ\phialp^{-1})_{\phialp}^{(1)}\bfV\phialp\right] \\
  & = & \textstyle\bfU\bfV\phialp + \frac{\alpha\mp 1}{2}\bfU\phipo\cdot\bfV\phialp.
\end{eqnarray*}

\subsection{The $\alpha$-divergences} \label{se:alpdiv}

These are defined on $M$ as follows. (Cf.~section 3.6 in \cite{amar1}.)
\begin{eqnarray}
\cdmo(P\,|\,Q)
  & = & \cdpo(Q\,|\,P) \nonumber \\
        [-1.5ex] \label{eq:alpdiv1} \\ [-1.5ex]
  & = & \Emu(\phimo(Q)-\phimo(P))
        + \langle\phimo(P)+1, \phipo(P)-\phipo(Q)\rangle_\LL, \nonumber
\end{eqnarray}
and, for $\alpha\neq\pm 1$,
\begin{eqnarray}
\cdal(P\,|\,Q)
  & = & \textstyle \frac{2}{1+\alpha}\Emu(\phimo(P)-\phialp(P)) \nonumber \\
        [-1.5ex] \label{eq:alpdiv2} \\ [-1.5ex]
  &   & \quad\textstyle + \frac{2}{1-\alpha}\Emu(\phimo(Q)-\phimalp(Q))
        - \langle\phialp(P), \phimalp(Q)\rangle_\LL. \nonumber 
\end{eqnarray}
It follows from Proposition \ref{pr:nemexp} that $\cdal\in C^\infty(M\times M;\R)$.
The following proposition summarises some other properties.

\begin{proposition} \label{pr:legend}
For any $\alpha\in\R$:
\begin{enumerate}
\item[(i)] $\cdmal(P\,|\,Q)=\cdal(Q\,|\,P)\ge 0$, with equality if and only if $P=Q$;
\item[(ii)] the following {\em generalised cosine rule} applies
\begin{eqnarray}
\cdal(P\,|\,R)
  & = & \cdal(P\,|\,Q) + \cdal(Q\,|\,R) \nonumber \\
        [-1.5ex] \label{eq:cosine} \\ [-1.5ex]
  &   & \quad - \langle \phialp(P)-\phialp(Q), \phimalp(R)-\phimalp(Q)\rangle_\LL;
        \nonumber
\end{eqnarray}
\item[(iii)] the set $\phialp(M)$ is convex;
\item[(iv)] for any $Q\in M$, the function
$\cdal(\phialp^{-1}|\,Q):\phialp(M)\rightarrow \R$ admits the following derivatives
\begin{eqnarray}
\cdal(\phialp^{-1}|\,Q)_a^{(1)}u
  & = & \langle \phimalp\circ\phialp^{-1}(a)-\phimalp(Q), u \rangle_\LL
        \label{eq:divder1} \\
\cdal(\phialp^{-1}|\,Q)_a^{(2)}(u, v)
  & = & \Emu\exp_\bG(\alpha\phipo\circ\phialp^{-1}(a))\cdot u\cdot v; \label{eq:divder2}
\end{eqnarray}
in particular $\cdal(\phialp^{-1}|\,Q)$ is strictly convex;
\item[(v)] for any $Q\in M$, $a\in\phimalp(M)$,
\begin{eqnarray}
\cdmal(\phimalp^{-1}(a)\,|\,Q)
  & = & \max_{b\in\phialp(M)}\big\{\langle a-\phimalp(Q), b-\phialp(Q)\rangle_\LL
        \nonumber \\
        [-1.6ex] \label{eq:legend} \\ [-1.6ex]
  &   & \qquad\qquad\quad - \cdal(\phialp^{-1}(b)\,|\,Q)\big\}, \nonumber
\end{eqnarray}
and the unique maximiser is $\phialp\circ\phimalp^{-1}(a)$.
\end{enumerate}
\end{proposition}

\begin{proof}
Parts (i), (ii) and (iv) can be proven by straightforward calculations.  Part (iii) is
trivial when $\alpha=1$, since $\phipo(M)=\bG$.   Suppose, then, that
$\alpha\in\R\setminus\{1\}$.  For any distinct $P_0,P_1\in M$, and any $t\in(0,1)$, let
$a_t := (1-t)\phialp(P_0) + t\phialp(P_1)$; then we can define
\[
\textstyle p_t
  = \left(1+\frac{1-\alpha}{2}a_t\right)^{2/(1-\alpha)}
  = \left((1-t)p_0^{(1-\alpha)/2}+tp_1^{(1-\alpha)/2}\right)^{2/(1-\alpha)},
\]
where $p_i:=dP_i/d\mu$, $i=0,1$.
Since the infimum (over $x\in B$) of the term in brackets on the right-hand side here
is strictly positive, $\log p_t$ is well defined and bounded.  $p_t$ is thus the density
of a measure $P_t\in M$, and $\phialp(P_t)=a_t$, which completes the proof of part (iii).

Let $a\in\phimalp(M)$, and let $f:\phialp(M)\rightarrow\R$ be defined as follows:
\begin{eqnarray*}
f(b) & = & \langle a-\phimalp(Q), b-\phialp(Q)\rangle_\LL-\cdal(\phialp^{-1}(b)\,|\,Q) \\
     & = & \cdal(Q\,|\,\phimalp^{-1}(a)) - \cdal(\phialp^{-1}(b)\,|\,\phimalp^{-1}(a)) \\
     & = & \cdmal(\phimalp^{-1}(a)\,|\,Q) - \cdmal(\phimalp^{-1}(a)\,|\,\phialp^{-1}(b)),
\end{eqnarray*}
where we have used (\ref{eq:cosine}) in the second step and part (i) in the
third step.  Part (v) now follows from part (i).
\end{proof}

It follows from (\ref{eq:riemet}), (\ref{eq:divder2}) and (\ref{eq:transmap}) that, for
any $P\in M$ and $U, V\in T_PM$,
\begin{equation}
\langle U, V \rangle_P
  = \cdal(\phialp^{-1}| P)_{\phialp(P)}^{(2)}(U\phialp, V\phialp), \label{eq:methes}
\end{equation}
confirming the Hessian nature of the metric.  Furthermore,
\[
\cdal(\phialp^{-1}|\phialp^{-1})_{a,a}^{(2,1)}(u,v;\cdot) \equiv 0,
\]
for all $(a,u), (a,v)\in\Phialp(TM)$. Since the metric is positive definite, the
only tangent vector $X\in T_PM$, for which
\begin{equation}
\langle X, W\rangle_P
   = \cdal(\phialp^{-1}|\phialp^{-1})_{a,a}^{(2,1)}(u, v; W\phialp)
     \quad {\rm for\ all\ }W\in T_PM, \label{eq:projder}
\end{equation}
is the zero vector, and this provides further justification for the definition of the
covariant derivative in (\ref{eq:covder}).  (Cf.~the finite-dimensional case in section
3.2 of \cite{amar1}.)

\section{The manifold of probability measures} \label{se:manprob}

Let $\bG_0:=\{a\in\bG:\Emu a=0\}$, let $N:=\{P\in M: P(B)=1\}$, and let
$\phim:N\rightarrow\bG_0$ be the restriction of $\phimo$ to $N$. $N$ is a statistical
manifold modelled on $\bG_0$, with global {\em mixture chart} $\phim$.  It is trivially
a $C^\infty$-embedded submanifold of $M$.  A tangent vector at $P\in N$ is a signed
measure in $T_PM$ that has total mass zero.  The tangent bundle, $TN$, is trivialised by
the global chart $\Phim:TN\rightarrow\bG_0\times\bG_0$, defined to be the restriction of
$\Phimo$ to $TN$.  We can also define an {\em exponential chart},
$\Phie:TN\rightarrow\bG_0\times\bG_0$, as
follows:
\begin{equation}
\Phie(P,U)
  = (\phie(P),U\phie)
  = \left(\phipo(P)-\Emu\phipo(P),\, dU/dP-\Emu dU/dP\right). \label{eq:Phiedef}
\end{equation}

$M$ and $N$ are connected by the {\em normalisation map} $\nu:M\rightarrow N$
($\nu(P):=P/P(B)$) and the {\em inclusion map}, $\imath:N\rightarrow M$.  The
associated tangent maps, $T\nu$ and $T\imath$, have particularly simple
representations in terms of the charts $\Phipo$ and $\Phie$:
\begin{eqnarray}
\Phie\circ T\nu\circ\Phipo^{-1}(a,u)
  & = & (a-\Emu a, u-\Emu u) \nonumber \\
        [-1.5ex] \label{eq:norminc} \\ [-1.5ex]
\Phipo\circ T\imath\circ\Phie^{-1}(b,v)
  & = & \left(b-\log\Emu\exp_\bG(b), v-\EP v\right), \nonumber
\end{eqnarray}
where $\EP$ is expectation with respect to $P=\phie^{-1}(b)$. So
\begin{equation}
\Phipo\circ Ti\circ T\nu\circ \Phipo^{-1}(a, u) = (a-\log\Emu\exp_\bG(a), u-\EP u),
\end{equation}
where $P=\nu\circ\phipo^{-1}(a)$.  A tangent vector $V\in T_PM$ at $P\in N$ is in $T_PN$
if and only if $\EP V\phipo = \EP dU/dP = 0$.  So, for any $P\in N$, $U\in T_PM$ and
$V\in T_PN$,
\begin{eqnarray}
\langle U, V \rangle_P
  & = & \langle U\phipo, V\phipo\rangle_{L^2(P)}
        = \langle U\phipo-\EP U\phipo, V\phipo\rangle_{L^2(P)} \nonumber \\
        [-1.5ex] \label{eq:project} \\ [-1.5ex]
  & = & \langle T\imath\circ T\nu U\phipo, V\phipo\rangle_{L^2(P)}
        = \langle T\nu U, V\rangle_P, \nonumber
\end{eqnarray}
which shows that $T\nu U$ is the projection of $U$ onto $T_PN$ in the metric of
(\ref{eq:riemet}). (This corresponds, in the $\phipo$-representation, to projection from
$L^2(P)$ onto the subspace of functions with $P$-mean zero.) More generally,
$T\nu$ effects $1$-parallel transport of tangent vectors from $P\in M$ to
$\nu(P)\in N$, followed by projection onto $T_{\nu(P)}N$.

The {\em Fisher-Rao metric} on $TN$ is the restriction of the metric of (\ref{eq:riemet})
to $TN$:
\begin{eqnarray}
\langle U,V\rangle_P
  & = & \langle U\phipo, V\phimo\rangle_\LL
        = \langle U\phipo - \Emu U\phipo, V\phim\rangle_\LL \nonumber \\
        [-1.5ex] \label{eq:fishrao} \\ [-1.5ex]
  & = & \langle U\phie, V\phim\rangle_\LL. \nonumber
\end{eqnarray} 
The $\alpha$-covariant derivative on $TN$ is the projection of that defined on $TM$ in
(\ref{eq:covder}); for any $\bfU,\bfV\in\Gamma TN$,
\begin{equation}
\nabla_\bfU^\alpha\bfV
  = T\nu\circ\Phialp^{-1}(\phialp\circ\imath, \bfU\bfV(\phialp\circ\imath)).
    \label{eq:Ncovdef}
\end{equation}

\begin{proposition} \label{pr:Ncovder}
For any $\alpha\in\R$ and any $\bfU,\bfV\in\Gamma TN$,
\begin{eqnarray}
\nabla_\bfU^\alpha\bfV\phie
  & = & \textstyle\bfU\bfV\phie + \frac{1-\alpha}{2}\big[
        (\bfU\phie-\EP\bfU\phie)\cdot(\bfV\phie-\EP\bfV\phie) \nonumber \\
        [-1.5ex] \label{eq:Ncovder} \\ [-1.5ex]
  &   & \qquad\qquad\qquad - \Emu(\bfU\phie-\EP\bfU\phie)\cdot(\bfV\phie-\EP\bfV\phie) 
        \big]. \nonumber
\end{eqnarray}
\end{proposition}

\begin{proof}
Let $\bfW\in\Gamma TM$ be defined by $\bfW_P\phialp=\bfU_{\nu(P)}\bfV(\phialp\circ\imath)$.
According to (\ref{eq:transmap}) and (\ref{eq:norminc}), for any $P\in N$,
\begin{eqnarray*}
\bfW_P\phipo
  & = & (\phipo\circ\phialp^{-1})_{\phialp(P)}^{(1)}
        \bfU_P\left[(\phialp\circ\phipo^{-1})_{\phipo}^{(1)}
        (\phipo\circ\phie^{-1})_{\phie}^{(1)}\bfv_e\right] \\
  & = & \textstyle\exp_\bG\left(\frac{\alpha-1}{2}\phipo\right)\cdot\bfU\left[
        \exp_\bG\left(\frac{1-\alpha}{2}\phipo\right)\cdot(\bfv_e-\E_{\id_N}\bfv_e)
          \right] \\
  & = & \textstyle\frac{1-\alpha}{2}\bfU\phipo\cdot(\bfv_e-\EP\bfv_e)
          + \bfU(\bfv_e-\E_{\id_N}\bfv_e) \\
  & = & \textstyle\frac{1-\alpha}{2}(\bfu_e-\EP\bfu_e)\cdot(\bfv_e-\EP\bfv_e)
          + \bfU(\bfv_e-\E_{\id_N}\bfv_e),
\end{eqnarray*}
where $\bfu_e:=\bfU\phie$ and $\bfv_e:=\bfV\phie$.  Now
$\nabla_\bfU^\alpha\bfV=T\nu\bfW$, and so
$\nabla_\bfU^\alpha\bfV\phie = \bfW\phipo - \Emu\bfW\phipo$, which completes the proof.
\end{proof}

\begin{remark} \label{re:actens}
The {\em Amari-Chentsov tensor} on $N$ is the symmetric covariant 3-tensor field $\tau$,
defined by 
\begin{equation}
\tau_P(U,V,W) = \E_P(u-\EP u)\cdot(v-\EP v)\cdot(w-\EP w), \label{eq:actens}
\end{equation}
where $U,V,W\in T_PN$, $u=U\phie$, $v=V\phie$ and $w=W\phie$. As in the
finite-dimensional case,
\begin{equation}
\cdal(\phie^{-1}|\phie^{-1})_{a,a}^{(2,1)}(u, v; w)
  = - \frac{1-\alpha}{2}\tau_P(U,V,W), \label{eq:derdivN}
\end{equation}
where $a=\phie(P)$, and this can be used to define the $\alpha$-covariant derivative
directly on $N$.
\end{remark}

It follows from (\ref{eq:Ncovder}) that an $\alpha$-geodesic of $N$ is a smooth curve
$\bfP$ satisfying the differential equation
\begin{equation}
\phie(\bfP)^{\prime\prime}
  = -\textstyle \frac{1-\alpha}{2}\left[(\phie(\bfP)^\prime-\E_\bfP\phie(\bfP)^\prime)^2
      - \Emu(\phie(\bfP)^\prime-\E_\bfP\phie(\bfP)^\prime)^2\right].  \label{eq:geoeqn}
\end{equation}
The Fenchel-Legendre transform of Proposition \ref{pr:legend} is preserved on $N$ when
$\alpha=\pm 1$; the role of the adjoint variables $\phipo$ and $\phimo$ is then played by
$\phie$ and $\phim$.

Setting $\alpha=1$ in (\ref{eq:Ncovder}), we see that $N$ is $1$-flat and that $\phie$
is an affine chart for $\nabla^1$.  Furthermore, it is clear that $N$ is also $-1$-flat
and that $\phim$ is an affine chart for $\nabla^{-1}$.  $N$ is thus {\em dually flat}
($\alpha=\pm 1$).  Its $-1$-flatness arises from the trivial nature of its embedding in
$M$ when expressed in terms of the chart $\phimo$; this is the natural embedding
of a set of probability measures in a linear space of signed measures.  Its $1$-flatness
is associated with its Lie group structure under the product $(PQ)_N:=\nu(PQ)$.

The products on $M$ and $N$ have practical significance as ``data fusion'' operators
in Bayesian estimation. Let $(\Omega,\ccf,\PR)$ be a probability space, on which
are defined random variables $X:\Omega\rightarrow B$ and $Y:\Omega\rightarrow\bY$, where
$(\bY,\cy,\lambda)$ is a measure space.  Let $P_{XY}$ and $P_X$ be the distributions of
$(X,Y)$ and $X$, respectively, and suppose that $P_X\in N$ and that there exists a
measurable function $\Lambda:B\times\bY\rightarrow (0,\infty)$ with the following
properties:
\begin{enumerate}
\item[(i)] $P_{XY}(A)=\int_A\Lambda\,d(P_X\otimes\lambda)$ for all $A\in\cx_B\times\cy$;
\item[(ii)] $\Lambda(\cdot, y)\in\bG^+$ for all $y\in\bY$.
\end{enumerate}
For each $y$, $\Lambda(\cdot,y)$ is a {\em likelihood function} in the Bayesian problem
of estimating $X$ on the basis of the prior distribution $P_X$ and the observation
$Y=y$.  Let $R:\bY\rightarrow M$ and $P_{X|Y}:\bY\rightarrow N$ be defined as follows:
\begin{eqnarray}
R(y)(A)
  & = & \int_A \Lambda(x,y)\,\mu(dx), \nonumber \\
        [-1.5ex] \label{eq:bayprod} \\ [-1.5ex]
P_{X|Y}(y)(A)
  & = & \nu(P_XR(y)) = (P_X\nu(R(y)))_N; \nonumber
\end{eqnarray}
then $P_{X|Y}(y)$ is a regular $(Y=y)$-conditional probability distribution for $X$.  (A
specific example is that in which $\bY=\R$, $Y=f(X)+\zeta$, $f\in\bG$, and $\zeta$ has
the standard Gaussian distribution and is independent of $X$.  The likelihood function
then takes the form $\Lambda(x,y)=n(0,1)(y-f(x))$, where $n(0,1)$ is the standard
Gaussian density.)

The space $\bY$ is unimportant in estimation problems of this type in the sense that
the observation $Y=y$ is represented by any of the surrogates
$\Lambda(\cdot,y)\in\bG^+$, $R(y)\in M$ or $\nu(R(y))\in N$.  Bayes' formula can thus
be interpreted in terms of products of measures in the commutative Banach algebras
$\bG$ and $\bG_0$, making it the classical equivalent of the ``collapse of state'' of
quantum probability, under observation.  The latter involves the normalised product
(iteration) of a quantum state (positive semi-definite operator) and a real-valued
observable (Hermitian operator), both acting on an underlying Hilbert space
\cite{meye1}.

Straightforward calculations show that, for any $\alpha\in\R$ and any
$\bfU, \bfV, \bfW\in\Gamma TN$,
\[
\nabla_\bfU^\alpha\nabla_\bfV^\alpha\bfW - \nabla_\bfV^\alpha\nabla_\bfU^\alpha\bfW
  - \nabla_{[\bfU, \bfV]}^\alpha\bfW
  = \textstyle\frac{1-\alpha^2}{4}\left(\langle \bfV, \bfW\rangle_P\bfU-\langle \bfU,
    \bfW\rangle_P\bfV\right).
\]
The {\em curvature tensor} for the $\alpha$-covariant derivative, $R^\alpha$, is
thus the covariant 4-tensor field defined by
\begin{equation}
R_P^\alpha(U, V, W, X)
  = \textstyle\frac{1-\alpha^2}{4}\left(\langle V, W\rangle_P\langle U, X\rangle_P
    -\langle U, W\rangle_P\langle V, X\rangle_P\right), \label{eq:curten}
\end{equation}
where $U,V,W,X\in T_PN$.  This is equal for the dual covariant derivatives
$\nabla^{\pm\alpha}$, and is zero if and only if $\alpha=\pm 1$.

Since all pairs $P,Q\in N$ are mutually absolutely continuous with relative densities
in $\bG^+$, any $P\in N$ can assume the role of $\mu$ in the construction of
{\em local charts}, $\phi_{m,P}, \phi_{e,P}:N\rightarrow \bG_{0,P}$, where
\begin{eqnarray*}
\phi_{m,P}(Q)
  & := & dQ/dP - 1, \\
\phi_{e,P}(Q)
  & := & \log_{\bG^+}(dQ/dP) - \E_P\log_{\bG^+}(dQ/dP),
\end{eqnarray*}
and $\bG_{0,P}:=\{a\in \bG: \E_P a=0\}$.  Such charts are {\em normal} for the
$\pm 1$-covariant derivatives at $P$, in the sense that, for any $U,V\in T_PN$
\[
\langle U\phi_{m,P}, V\phi_{m,P}\rangle_{L_0^2(P)}
  = \langle U\phi_{e,P}, V\phi_{e,P}\rangle_{L_0^2(P)}
  = \langle U, V\rangle_P,
\]
and, for any $\bfU,\bfV\in\Gamma TN$,
\begin{equation}
\nabla_\bfU^{-1}\bfV\phi_{m,P} = \bfU\bfV\phi_{m,P}\quad{\rm and}\quad
\nabla_\bfU^{+1}\bfV\phi_{e,P} = \bfU\bfV\phi_{e,P}. \label{eq:locnorm}
\end{equation}
In \cite{loqu1}, G.~Loaiza and H.R.~Quiceno developed a {\em $q$-deformed exponential}
statistical manifold by introducing local charts at each point.  These
are based on the Tsallis $q$-logarithm ($0<q<1$).  In the context of $N$, as developed
here, these local charts, $s_{q,P}:N\rightarrow \bG_{0,P}$, take the form
\begin{equation}
s_{q,P}(Q)
  := \frac{\phi_{\alpha,P}(Q)-\E_P\phi_{\alpha,P}(Q)}{1+(1-q)\E_P\phi_{\alpha,P}(Q)},
     \label{eq:loquchar}
\end{equation}
where $\alpha=2q-1$ and $\phi_{\alpha,P}:M\rightarrow\bG$ is the local
variant of $\phialp$
\[
\phi_{\alpha,P}(Q):=\textstyle\frac{2}{1-\alpha}\left((dQ/dP)^{(1-\alpha)/2}-1\right).
\]
They are ``locally normal'' for the $\alpha$-covariant derivative at $P$, in the sense
that, for any $U,V\in T_PN$,
$\langle Us_{q,P}, Vs_{q,P}\rangle_{L_0^2(P)}=\langle U, V\rangle_P$ and, for any
$\bfU,\bfV\in\Gamma TN$,
\[
\nabla_\bfU^{\alpha}\bfV s_{q,P}(P) = \bfU\bfV s_{q,P}(P).
\]
Unlike the charts $\phi_{m,P}$ and $\phi_{e,P}$, for which (\ref{eq:locnorm})
is true on the whole of $N$, this is true only at $P$, reflecting the non-zero
curvature of the $\alpha$-connection at $P$.

\section{Manifolds of smooth densities} \label{se:smooth}

In this section we consider the {\em sequences} of manifolds $(M^k, k\in\N_0)$ and
$(N^k,k\in\N_0)$, as developed in sections \ref{se:manfin} and \ref{se:manprob}, making
explicit their dependence on the number of derivatives in the definition of $\bG$
($=\bG^k$).  By developing projective limits of these sequences, we define Fr\'{e}chet
manifolds of measures having smooth densities with respect to $\mu$.  The manifold of
finite measures in this context, its tangent bundle, and its model space, are as follows:
\begin{equation}
\textstyle \Mbar := \bigcap_{k\in\N_0}M^k, \quad
  T\Mbar := \bigcap_{k\in\N_0} TM^k\quad{\rm and}\quad
  \bGbar := \bigcap_{k\in\N_0} \bG^k.  \label{eq:Minfdef}
\end{equation}
Let $\rho^k:\bGbar\rightarrow\bG^k$ be the inclusion map. $\bGbar$ is a Fr\'{e}chet
space, whose topology is generated by the sequence of norms
$(\|\rho^k\|_{\bG^k}, k\in\N_0)$; $\Mbar$ is a Fr\'{e}chet manifold of finite measures
on $(B,\cx_B)$, whose densities with respect to $\mu$ are smooth bounded functions, having
bounded derivatives of all orders.  Its tangent bundle is trivialised by the chart
$\Phibar_1:T\Mbar\rightarrow\bGbar\times\bGbar$, which is defined to be the restriction
of $\Phipo^k$ to $T\Mbar$.  A tangent vector at $\Pbar\in\Mbar$ is a signed measure on
$(B,\cx_B)$, whose density with respect to $\Pbar$ is of the form $d\Ubar/d\Pbar=u$ for
some $u\in\bGbar$.

A map $f:\bGbar\rightarrow\bFbar$, taking values in another Fr\'{e}chet space $\bFbar$,
is said to be {\em Leslie differentiable} \cite{lesl1}, with derivative
$df:\bGbar\rightarrow L(\bGbar;\bFbar)$, if, for any $a\in\bGbar$, the map
$R_a:\R\times\bGbar\rightarrow\bFbar$ defined by
\begin{equation}
R_a(t,u) := \left\{
  \begin{array}{ll}
    t^{-1}(f(a+tu)-f(a))-df(a)u & {\rm if\ }t\ne 0 \\
    0 & {\rm if\ }t=0,
  \end{array}\right. \label{eq:lesdef}
\end{equation}
is continuous at $(0,u)$ for every $u\in\bGbar$.  The study of the Leslie
differentiability properties of a map between Fr\'{e}chet spaces (including the regularity
of its derivatives, considered as maps into spaces of continuous linear maps) becomes
substantially easier if the map in question is the projective limit of a system of
maps between Banach spaces \cite{dogv1}.

For any $0\le j\le k<\infty$, let $\rho^{kj}:\bG^k\rightarrow\bG^j$ be the (continuous
linear) inclusion map.  The system $(\bG^k, \rho^{kj}, 0\le j\le k<\infty)$ is a
{\em projective system} with {\em factor spaces} $\bG^k$ and {\em connecting morphisms}
$\rho^{kj}$.  The {\em projective limit} of this system is the following subset of the
cartesian product $\Pi:=\prod_{k=0}^\infty \bG^k$:
\begin{equation}
\varprojlim \bG^k
  := \left\{(a^0, a^1, \ldots)\in\Pi: \rho^{kj}a^k=a^j
     {\rm\ for\ all\ }0\le j\le k<\infty \right\}. \label{eq:plimg}
\end{equation}
In this particular example, the map
$\varprojlim \bG^k\ni (\rho^0\abar, \rho^1\abar, \ldots)\mapsto \abar\in\bGbar$ is a
toplinear isomorphism, and so we can identify $\varprojlim \bG^k$ with $\bGbar$.  The
inclusion map $\rho^k:\bGbar\rightarrow\bG^k$ then plays the role of the {\em canonical
projection} \cite{dogv1}.

Suppose that $(\bF^k, \sigma^{kj}, 0\le j\le k<\infty)$ is another projective system of
Banach spaces with projective limit $\bFbar$.  The sequence
$(f^k:\bG^k\rightarrow\bF^k, k\in\N_0)$ is a {\em projective system of maps} if
\begin{equation}
\sigma^{kj}f^k
  = f^j\rho^{kj}\quad{\rm for\ all\ }0\le j\le k<\infty. \label{eq:plimmap}
\end{equation}
The projective limit of this system is $\fbar:\bGbar\rightarrow\bFbar$,
defined by $\fbar(\abar)=(f^0(\abar), f^1(\abar), \ldots)$.  If each $f^k$ is
(Fr\'{e}chet) differentiable then $\fbar$ is Leslie differentiable, and its derivative
can be associated with a projective limit of those of $f^k$.  (See Proposition 2.3.11
in \cite{dogv1}.)  The appropriate projective system of derivatives is
$(\Delta f^k:\bG^k\rightarrow H^k(\bG;\bF),\;k\in\N_0)$, where
\begin{equation}
\Delta f^k
  := \big(f_{\rho^{k0}}^{0,(1)}, f_{\rho^{k1}}^{1,(1)},\ldots,f^{k,(1)}\big),
     \label{eq:dersys}
\end{equation}
and
\begin{equation}
H^k(\bG;\bF)
  = \bigg\{(\lambda^0, \ldots, \lambda^k)\in\prod_{i=0}^k L(\bG^i;\bF^i):
    \sigma^{ji}\lambda^j=\lambda^i\rho^{ji},\; i\le j\bigg\}. \label{eq:hkspace}
\end{equation}
The factor spaces $H^k(\bG;\bF)$ are connected by the morphisms
$h^{kj}:H^k(\bG;\bF)\rightarrow H^j(\bG;\bF)$,
\[
h^{kj}(\lambda^0, \ldots, \lambda^k) = (\lambda^0, \ldots, \lambda^j),\quad j\le k,
\]
and so constitute a projective system of Banach spaces.  The associated projective limit
is toplinear isomorphic with $\Hbar(\bG;\bF)$ (defined by the obvious variant of
(\ref{eq:hkspace})), and the map $\epsilon:\Hbar(\bG;\bF)\rightarrow L(\bGbar;\bFbar)$,
defined by $\epsilon(\lambda^0, \lambda^1, \ldots)=\varprojlim\lambda^k
=(\lambda^0\rho^0, \lambda^1\rho^1, \ldots)$, is continuous linear, with respect to
the toplogy of uniform convergence on bounded sets. (See Theorem 2.3.10
in \cite{dogv1}.)  That $\Hbar(\bG;\bF)$ is a projective limit of Banach spaces is
central to the regularity of the Leslie derivative $d\fbar$.  If each $f^k$ is smooth
then $d\fbar:\bGbar\rightarrow L(\bGbar;\bFbar)$ is Leslie smooth.  (See Propositions
2.3.11, 2.3.12 in \cite{dogv1}.)

Applying these ideas to the transition maps $\Phialp^k\circ(\Phipo^k)^{-1}$ and their
inverses, we see that the projective limit $\Phibar_\alpha\circ\Phibar_1^{-1}$ is Leslie
diffeomorphic, and all its derivatives (together with those of its inverse) are smooth
maps from open subsets of $\bGbar$ to appropriate spaces of continuous linear maps.
$(\Phibar_\alpha, \alpha\in\R)$ is a Leslie smooth atlas for $T\Mbar$.  For any Leslie
differentiable map $f:\bGbar\rightarrow\bFbar$ and any $\Ubar\in T_\Pbar\Mbar$
\[
\Ubar f = d(\fbar\circ\phibar_\alpha^{-1})\Ubar\phibar_\alpha\quad{\rm for\ any\ }
           \alpha\in\R.
\]

We define a special class of smooth vector fields of $\Mbar$ -- those whose
$\Phibar_\alpha$-representations are projective limits of smooth maps between the Banach
spaces $\bG^k$.  Let $S$ be the following set of sequences:
\begin{equation}
S = \big\{(\bfn_k\in\N_0, k\in\N_0): \bfn_k\le \bfn_{k+1},\; \sup\bfn_k=+\infty\big\},
    \label{eq:indcond}
\end{equation}
and note that, for any $\bfn\in S$,
$(\bG^{\bfn_k}, \rho^{\bfn_k\bfn_j},\; 0\le j\le k<\infty)$ is a projective system of
Banach spaces with projective limit $\bGbar$.  For some $\bfn\in S$, let
$(\bfu^k:\bG^k\rightarrow\bG^{\bfn_k},\; k\in\N_0)$ be a projective system of
smooth maps, with projective limit $\bfubar:\bGbar\rightarrow\bGbar$.  We regard
$\bfubar\circ\phibar_1$ as being the $\Phibar_1$-representation of a smooth vector
field $\bfUbar$ ($\bfUbar\phibar_1:=\bfubar\circ\phibar_1$).  We denote the set of all
such {\em projective-limit smooth vector fields} $\Gamma_{\rm pl}T\Mbar$.  This has a
linear structure, in which the sum of $(\bfu^k:\bG^k\rightarrow\bG^{\bfn_k})$ and
$(\bfv^k:\bG^k\rightarrow\bG^{\bfm_k})$, for $\bfm,\bfn\in S$, is the projective system
$(\bfw^k:=\rho^{\bfn_k\bfl_k}\bfu^k+\rho^{\bfm_k\bfl_k}\bfv^k:
\bG^k\rightarrow\bG^{\bfl_k})$, where $\bfl_k:=\min\{\bfm_k,\bfn_k\}$.

\begin{remark} \label{re:gampl}
$\Gamma_{\rm pl}T\Mbar$ is strictly smaller than $\Gamma T\Mbar$ -- it does not contain
the vector field with $\Phibar_1$-representation $\bfubar(\abar)=\abar r(\abar,0)$, for
example, where $r$ is the usual metric on $\bGbar$.  However, it does contain many
useful vector fields occurring in the theory of partial differential equations.  For
example, if $\bX=\R^d$ then the second-order differential operator
$\partial^2/\partial x_i\partial x_j$ lifts to a vector field in $\Gamma_{\rm pl}T\Mbar$.
(Cf.~(\ref{eq:kusstra}).)
\end{remark}

\begin{proposition} \label{pr:fdiff}
Let $\bfubar:\bGbar\rightarrow\bGbar$ be as defined above, and let
$(f^k:\bG^k\rightarrow\bF^k)$ be a projective system of {\em smooth} maps, as described
in (\ref{eq:plimmap}).  Then the sequence of maps
\begin{equation}
\left(f_{\rho^{kl}}^{l,(1)}\rho^{\bfn_kl}\bfu^k:\bG^k\rightarrow\bF^l,\;
l=\min\{k,\bfn_k\},\; k\in\N_0\right) \label{eq:seqder}
\end{equation}
is a projective system, with projective limit $d\fbar\bfubar$, and
$\bfUbar(\fbar\circ\phibar_1)=d\fbar\bfubar\circ\phibar_1$.
\end{proposition}

\begin{proof}
For any $j\le k$, let $l:=\min\{k,\bfn_k\}$ and $m:=\min\{j,\bfn_j\}$.
Differentiating the projective relation $\sigma^{lm}f^l=f^m\rho^{lm}$, we obtain
$\sigma^{lm}f^{l,(1)}=f_{\rho^{lm}}^{m,(1)}\rho^{lm}$.
Restricting the base-point from $\bG^l$ to $\bG^k$, and applying the resulting linear
map to $\rho^{\bfn_kl}\bfu^k$, we obtain
\[
\sigma^{lm}f_{\rho^{kl}}^{l,(1)}\rho^{\bfn_kl}\bfu^k
  = f_{\rho^{km}}^{m,(1)}\rho^{\bfn_km}\bfu^k
  = \left(f_{\rho^{jm}}^{m,(1)}\rho^{\bfn_jm}\bfu^j\right)\rho^{kj},
\]
which establishes the projective property.  The projective limit is
\[
\left(f_{\rho^{\bfl_0}}^{\bfl_0,(1)}\rho^{\bfl_0}\bfubar,
  f_{\rho^{\bfl_1}}^{\bfl_1,(1)}\rho^{\bfl_1}\bfubar, \ldots\right)
  \equiv d\fbar\bfubar,
\]
where $\bfl_k:=\min\{k,\bfn_k\}$.  Let $\bfP\in\bfUbar(\Pbar)$; then
\[
\bfUbar(\Pbar)(\fbar\circ\phibar_1)
  = (\fbar\circ\phibar_1(\bfP))^\prime(0)
  = d\fbar\bfUbar(\Pbar)\phibar_1
  = d\fbar\bfubar\circ\phibar_1(\Pbar),
\]
which completes the proof.
\end{proof}

Suppose that $\bfVbar\in\Gamma_{\rm pl}T\Mbar$ is defined by the projective system of
smooth maps $(\bfv^k:\bG^k\rightarrow\bG^{\bfm_k},\; k\in\N_0)$ for some $\bfm\in S$.
By applying Proposition \ref{pr:fdiff} to the projective system
$(\phialp^{\bfm_k}\circ(\phipo^{\bfm_k})^{-1}\circ\bfv^k:
\bG^k\rightarrow\bG^{\bfm_k}(=:\bF^k))$, we can define the $\alpha$-covariant derivative
on $\Mbar$: $\nabla_\bfUbar^\alpha \bfVbar\phibar_1=\bfwbar\circ\phibar_1$, where
\begin{equation}
\bfwbar
  = d(\phibar_1\circ\phibar_\alpha^{-1})
    d(d(\phibar_\alpha\circ\phibar_1^{-1})\bfvbar)\bfubar
  = d\bfvbar\bfubar + \textstyle\frac{1-\alpha}{2}\bfvbar\cdot\bfubar.  \label{eq:cdlim}
\end{equation}
The $\Phibar_1$-representation $\bfwbar$ is the projective limit of the system
$(\bfw^k:\bG^k\rightarrow\bG^{\bfi_k},\;\; k\in\N_0)$, where
$(\bfi_k=\min\{\bfm_k,\bfn_k, \bfm_{\bfn_k}\},\; k\in\N_0)\in S$, and so
$\nabla_\bfUbar^\alpha \bfVbar\in\Gamma_{\rm pl}T\Mbar$.
The remaining constructions in sections \ref{se:manfin} and \ref{se:manprob} carry over
to $\Mbar$ without difficulty.  Key points to note are as follows.
\begin{enumerate}
\item[$\bullet$] The smoothness of the $\alpha$-divergences on $\Mbar$ follows from
their smoothness on $M^k$, and that of the inclusion map $\imath^k:\Mbar\rightarrow M^k$. 
The metric and covariant derivatives could be derived directly from $\cdal$ as in sections
\ref{se:manfin} and \ref{se:manprob}.

\item[$\bullet$] The statistical manifold $\Nbar$ is defined in the obvious way.  It is
a Leslie $C^\infty$-embedded submanifold of $\Mbar$ since its image under
$\phibar_{-1}$ is the subspace of $\bGbar$ comprising those members with zero $\mu$-mean.

\item[$\bullet$] An $\alpha$-geodesic of $\Nbar$ is a smooth curve $\bfP$ whose projection
$\imath^k\bfP$ satisfies (\ref{eq:geoeqn}) for all $k$.  ($\alpha$-geodesics of $\Mbar$,
and $\pm 1$-geodesics of $\Nbar$ are, of course, straight lines in appropriate charts.)
\end{enumerate}

\section{Concluding Remarks}

In this paper we have constructed a family of non-parametric statistical manifolds, $N$,
that support the full geometry of the Fisher-Rao metric and Amari $\alpha$-covariant
derivatives for all $\alpha\in\R$.  Manifolds in the family admit global mixture and
exponential charts, $\phim$ and $\phie$, which are of importance in applications.  The
$\alpha$-covariant derivatives were computed in the chart $\phie$, and their curvature
on $N$ was found: $N$ is $\alpha$-flat if and only if $\alpha=\pm 1$; otherwise the
curvature tensor changes sign as $|\alpha|$ passes through $1$.  The $\alpha$-divergences
are of class $C^\infty$ on $N$. As in the parametric case, their derivatives provide an
alternative way of defining covariant derivatives.

The statistical manifolds were constructed {\em extrinsically}, via smooth embeddings in
particular manifolds of finite measures, $M$.  The latter are covered by every chart in
a smooth one-parameter atlas $(\phialp, \alpha\in\R)$, each chart inducing its own
parallel transport on the tangent bundle.  Since $M=\phipo^{-1}(\bG)$ and
$N=\phie^{-1}(\bG_0)$, the manifolds $M$ and $N$ are, in one sense, no more than linear
spaces themselves.  The statistical interest comes from the interplay between the
different affine representations provided by the charts $\phialp$, in particular
$\phipo$ and $\phimo$.  Manifold theory enters the picture with the introduction of the
base-point-dependent Fisher-Rao metric on the tangent bundle.

The manifolds are applicable to problems in Bayesian estimation, in which the
$\alpha$-divergences are important measures of approximation error.  In particular,
the process of posterior distributions of a nonlinear filter can be expressed as a
solution of a suitable differential equation on $N$.  In this, and other potential
applications to Physics (eg.~the Fokker-Planck or the Boltzmann equations), it is
important for the members of $N$ to have differentiable densities, and this is built in
to its definition.  A Fr\'{e}chet manifold of probability measures having {\em smooth}
densities was defined in section \ref{se:smooth}, via projective limits.
Under suitable technical conditions, the coefficients of certain partial differential
equations, including those of nonlinear filtering, can be interpreted as vector fields
of this manifold.

The extra regularity of $N$, over other non-parametric manifolds in the literature, is
gained at the cost of inclusiveness: each probability density in $N$ has a strictly
positive infimum.  Although many idealised models of physical systems do not exhibit
this property, it is often possible to substitute a model that does.  In the context of
nonlinear filtering of diffusion processes, for example, we can constrain the diffusion
to stay on a bounded domain by introducing suitable boundary conditions.  It is open to
question whether this results in any less accurate a model of reality than the idealised
model.  This is a question for future research.  Another potential avenue is to use $N$
as a starting point in the construction of manifolds that place less stringent
regularity constraints on their members.

\section*{Acknowledgements}

The author would like to thank both anonymous referees for their careful reading of the
paper. Their comments have led to a substantial improvement in its presentation.

\end{document}